\documentclass[a4paper,12pt,oneside]{article}
\usepackage[top=3cm, bottom=3cm, left=3cm, right=3cm,dvipdfm]{geometry}
\usepackage[latin2]{inputenc}
\usepackage{amsmath,amssymb,amsthm}
\usepackage{thm-restate}
\usepackage[dvipdfm]{hyperref}
\usepackage{xcolor}

\hypersetup{
   colorlinks,
    linkcolor={blue!50!black},
    citecolor={red!40!black},
   urlcolor={green!60!black}
}
\usepackage{dsfont}
\usepackage{color}
\usepackage{graphicx}
\usepackage{enumerate}

\theoremstyle{plain}
\newtheorem{thm}{Theorem}[section]
\newtheorem{lemma}[thm]{Lemma}

\newtheorem{claim}[thm]{Claim}

\newtheorem{coroll}[thm]{Corollary}
\theoremstyle{definition}
\newtheorem{constr}{Construction}[section]

\newtheorem{conj}{Conjecture}

\theoremstyle{remark}

\def\inte{1{\text{\rm -int}}}
\def\lii{01{\text{\rm -int}}}
\def\cro{1}
\def\width{thickness }
\def\restricted{bounded }
\def\WWW{{\mathcal{H}}}   % (3 darab)

\def\qed{\nobreak\hfill $\Box$\par\smallskip}%
\def\AA{{\mathcal{A}}}%
\def\BB{{\mathcal{B}}}%
\def\CC{{\mathcal{C}}}%
\def\RR{{\mathbb{R}}}%
\def\cL{{\mathcal{L}}}%
\def\SS{{ (\mathcal{A},\mathcal{B})}}%
\def\HH{{\mathcal{H}}}%
\def\beq{\begin{equation}}\def\eeq{\end{equation}}

\title{Problems and results on 1-cross intersecting set pair systems}

\author{Zolt\'an F\"uredi\thanks{Research was supported in part by NKFIH grant KH130371 and
  NKFI--133819.}, \,
Andr\'as Gy\'arf\'as\thanks{Research was supported in part by
NKFIH grant K116769.}\\[-0.8ex]
\small Alfr\'ed R\'enyi Institute of Mathematics\\[-0.8ex]
\small P.O. Box 127\\[-0.8ex]
\small Budapest, Hungary, H-1364\\[-0.8ex] \small
\texttt{furedi.zoltan@renyi.hu, gyarfas.andras@renyi.hu} \and Zolt\'an
Kir\'aly\thanks{This research was partially supported by the Hungarian
  National Research, Development and Innovation Office, OTKA grant no. FK
  132524 and by Dynasnet European Research Council Synergy project (ERC-2018-SYG 810115).} \\[-0.8ex]
\small ELTE E\"otv\"os Lor\'and University \\[-0.8ex]
\small Department of Computer Science\\[-0.8ex]
\small P\'azm\'any P\'eter s\'et\'any 1/C \\[-0.8ex]
\small Budapest, Hungary, H-1117\\[-0.8ex]
\small and Alfr\'ed R\'enyi Institute of Mathematics\\[-0.8ex]
\small \texttt{kiraly@cs.elte.hu}\\
}
\begin{document}
\pagestyle{myheadings} \markright{{\small{\sc F\"uredi, Gy\'arf\'as, Kir\'aly:  1-cross intersecting set pair systems
}}}
\maketitle

\begin{abstract}
The notion of cross intersecting set pair system of size $m$, $\Big(\{A_i\}_{i=1}^m, \{B_i\}_{i=1}^m\Big)$ with
$A_i\cap B_i=\emptyset$ and $A_i\cap B_j\ne\emptyset$, was introduced by %comes from a paper of
Bollob\'as and it became an important tool of extremal combinatorics. % graph and set theory.
% The basic
His classical result % of Bollob\'as
states that $m\le {a+b\choose a}$ if $|A_i|\le a$ and $|B_i|\le b$ for each $i$.

Our central problem is to see how this bound changes %improves
 with the additional condition $|A_i\cap B_j|=1$ for $i\ne j$.
% We call
Such a % set pair
system is called $1$-cross intersecting. We show that these systems are related to perfect graphs, clique partitions of graphs, and finite geometries.
We prove that their maximum size is
\begin{itemize}

\item at least $5^{n/2}$ for $n$ even, % $2\cdot 5^{(n-1)/2}$ for $n$ odd
 $a=b=n$,

\item  equal to
  $\bigl(\lfloor\frac{n}{2}\rfloor+1\bigr)\bigl(\lceil\frac{n}{2}\rceil+1\bigr)$
  if $a=2$ and $b=n\ge 4$,

\item at most $|\cup_{i=1}^m A_i|$,

\item  asymptotically $n^2$ if $\{A_i\}$ is a linear hypergraph ($|A_i\cap A_j|\le 1$ for $i\ne j$),

\item  asymptotically ${1\over 2}n^2$ if $\{A_i\}$ and $\{B_i\}$ are both
  linear hypergraphs.

\end{itemize}

\end{abstract}

\section{Introduction, results}

The notion of cross intersecting set pair systems was introduced by % comes from a paper of
Bollob\'as~\cite{BO} and it became a standard tool of extremal set theory.
Because of its importance there are many proofs (e.g., Lov\'asz~\cite{Lov}, Kalai~\cite{Kalai}) and generalizations (e.g., Alon~\cite{Alon}, F\"uredi~\cite{ZF26}).
For applications and extensions of the concept the surveys of F\"uredi~\cite{ZF70} and Tuza~\cite{TU1,TU2} are recommended.

A \emph{cross intersecting set pair system % {\rm (SPS)}
of size $m\ge 2$} consists of finite sets $A_1,\ldots,A_m$ %\subseteq\NN$
and $B_1,\ldots,B_m$ %\subseteq\NN$
 such that % with two basic properties:

$$A_i\cap B_i=\emptyset \mathrm{\ for\ every\ } 1\le i \le m,$$
$$A_i\cap B_j\ne \emptyset \mathrm{\ for\ every\ } 1\le i\ne j \le m.$$
We will consider % make some
 further constrains % restrictions
  but always keep these two basic properties.

Bollob\'as' theorem~\cite{BO} states that
\beq \label{boll} m\le {a+b\choose a} \eeq
must hold for any cross intersecting set pair system if we have % further require
$|A_i|\le a$ and $|B_i|\le b$ for each $i$. This size can be achieved by the {\em standard example}, taking all $a$-element sets of an $(a\!+\!b)$-element set for the $A_i$-s and their complements as $B_i$-s.

Let $\AA=\{A_i\}_{i=1}^m$ and
$\BB=\{B_i\}_{i=1}^m$.
The set pair system (SPS for short) is denoted by
$(\AA,\BB)=\{(A_i,B_i)\}_{i=1}^{m}$.
%%% The union of the two hypergraphs is denoted by $\HH=\AA\cup\BB$.
An SPS is $(a,b)$-\emph{\restricted} \!\!\! if $|A_i|\le a$ and $|B_i|\le b$
for each $i$.

An  SPS $\SS$ is \emph{$1$-cross intersecting} if $|A_i\cap B_j|=1$ for each $i\ne j$.
Our aim is to find good estimates for the size under this condition. This  leads to interesting but seemingly difficult problems.

Our results are summarized in the next five subsections. In two warm-up sections we show that an  $1$-cross intersecting $(n,n)$-\restricted  SPS $\SS$ can have exponential size and that its size is bounded by the sizes of the vertex sets of $\AA$ (and $\BB$). We show how the latter provides an alternate ending of Gasparian's proof of Lov\'asz's perfect graph theorem. The next two subsections present our main results:  sharp bound of the size in the $(2,n)$-\emph{\restricted}case (Theorem \ref{2n}) and asymptotically best bounds for the size in the $(n,n)$-\emph{\restricted} case when $\AA,\BB$ are linear (Theorem \ref{gup}) and when $\AA,\BB$ are $1$-intersecting (Theorem \ref{hup}).  Then we show the connection of $1$-cross intersecting SPS-s  with clique partition of graphs.

Although the main results of this article are about $1$-intersecting families, we propose the problem in a very general setting in Section \ref{notgen}.  The proof of the upper bounds are in Sections~\ref{1-cross},~\ref{linear}. The constructions giving the lower bounds are in Section \ref{lboundlin} and we conclude with some open problems.

\subsection{1-cross intersecting SPS of exponential sizes}\label{expsize}

% Our starting point is  the study of

 % Somewhat surprisingly,
A  $1$-cross intersecting $(n,n)$-\restricted  SPS can have exponential size.
 %%%  because of the following proposition.
\begin{restatable}{prop}{propmult}\label{mult}
  If there exist an  $(a_1,b_1)$-\restricted  1-cross intersecting SPS of size $m_1$
  and an $(a_2,b_2)$-\restricted  1-cross intersecting SPS of size $m_2$,
then an $(a_1\!+\!a_2, b_1\!+\!b_2)$-\restricted  1-cross intersecting SPS also exists of  size $m_1\cdot m_2$.
\end{restatable}

The proof of this, and most other proofs, are postponed to later sections.

Starting from the standard example (with $a=b=1$ and $m=2$),
Proposition~\ref{mult} yields an $(n,n)$-\restricted  $1$-cross intersecting
SPS of size $2^n$, exponential in $n$.
Define the $(2,2)$-\restricted  $1$-cross intersecting SPS, called
$\WWW(2,2)$, using the edges of a five-cycle and its complement. The five pairs are  $\left(\{i, i\!+\!1\},\{i\!+\!2, i\!+\!4\}\right)$
taken$\pmod 5$. Then Proposition~\ref{mult} gives the following.

\begin{coroll}\label{fexp}  There exists an $(n,n)$-\restricted  $1$-cross intersecting SPS of size $5^{n/2}$ if $n$ is even and of size $2\cdot 5^{(n-1)/2}$ if $n$ is odd.  \qed
\end{coroll}
This is the best lower bound we know. % n to us.
It remains a challenge to decrease essentially the upper bound $2n\choose n$ in (\ref{boll}) for an $(n,n)$-\restricted  $1$-cross intersecting SPS.

Corollary~\ref{fexp} gives a $(3,3)$-\restricted $1$-cross intersecting SPS of size 10, in fact two different ones, with
12 and with 15 vertices, depending on the order we apply Proposition~\ref{mult}.
We have a third example, the pairs $(\{i, i\!+\!1, i\!+\!2\}, \{i\!+\!3, i\!+\!6, i\!+\!9\})$ taken$\pmod{10}$ has 10 vertices.
Samuel Spiro (sspiro@ucsd.edu) informed us that his computer program successfully checked that 10 is indeed the largest size.

\subsection{1-cross intersecting SPS and perfect graphs}
One particular feature of a $1$-cross intersecting SPS $\SS$ is that its size is bounded by the sizes of the vertex sets of $\AA$ (and $\BB$). This can be considered as a variant of Fischer's inequality, and does not hold for general SPS.
% This follows from the following proposition.

\begin{restatable}{prop}{proplinf}\label{linf}
  Assume that $\SS$ is $1$-cross intersecting and $V:=\cup \AA$. Then
the characteristic vectors of the edges of $\AA$ are linearly independent in $\RR^V$.
\end{restatable}

A special case of Proposition~\ref{linf} relates to perfect graphs and can be used in Gasparian's proof~\cite{GAS, D}  of Lov\'asz's characterization~\cite{LO} of perfect graphs:
a graph $G$ is perfect if and only if
\begin{equation}\label{lchar}
|V(H)|\le \alpha(H)\omega(H)
\end{equation}
holds for all induced subgraphs $H$ of $G$.

To prove the nontrivial part, Gasparian showed that if a minimal imperfect graph $G$ would satisfy (\ref{lchar}) then there is a 1-cross intersecting SPS of size $m=\alpha(G)\omega(G)+1$  defined by independent sets and complete subgraphs of $G$.  By Proposition \ref{linf}, $|V(G)|\ge \alpha(G)\omega(G)+1$, contradicting (\ref{lchar}).

\subsection{$(2,n)$-\emph{\restricted} 1-cross intersecting SPS}

Here we state the best bound for the size of $(2,n)$-\emph{\restricted} 1-cross intersecting SPS showing that the main term of the upper bound $\frac{1}{2}(n+2)(n+1)$ in (\ref{boll}) can be halved.

\begin{restatable}{thm}{thmtwon}\label{2n}
  Let $n\ge 4$, and let $\SS$ be a $(2,n)$-\restricted
$1$-cross intersecting SPS of size $m$. Then  $$m\le \left(\left\lfloor\frac{n}{2}\right\rfloor+1\right)\left(\left\lceil\frac{n}{2}\right\rceil+1\right).$$
This bound is the best possible.
For $n=2, 3$ the exact values are $m=5, 7$.
\end{restatable}

\subsection{1-cross intersecting SPS in linear hypergraphs}\label{main}

A hypergraph $\HH$ is called \emph{linear} if the intersection of any two
different edges has at most one vertex.
$\HH$ is called $1$-{\em intersecting} if $|H\cap H'|=1$ for all $H, H'\in \HH$ whenever $H\neq H'$.

% Our first observation here is that if
If one of $\SS$, say $\cal{A}$, in an  SPS is
linear, then the size of this SPS is bounded by  $n^2+O(n)$ (without any assumption on $|B_i\cap B_j|,|A_i\cap B_j|$).

\begin{restatable}{prop}{propujup}\label{ujup}
Suppose that $(\AA,\BB)$ is an $(n,n)$-\restricted
 cross intersecting SPS of size $m$ such that $\AA$ is a linear hypergraph.  Then $m \le n^2 +n+1$.
% The size of an $(n,n)$-\restricted  cross intersecting SPS  $\SS$ with linear $\AA$ is at most $n^2+n+1$.
\end{restatable}

When $\AA$ and $\BB$ are both linear, and they form a  1-cross intersecting SPS then this  bound % of Proposition \ref{ujup}
 can be approximately halved.

\begin{restatable}{thm}{thmgup}\label{gup}
Suppose that $(\AA,\BB)$ is an $(n,n)$-\restricted
 1-cross intersecting SPS of size $m$ such that both $\AA$ and $\BB$ are linear hypergraphs.  Then $m \le \frac{1}{2}n^2 +n+1$.
% $7\le m_3(1,1,1)\le 9$.
\end{restatable}

A further small decrement comes if in addition $\AA$ and $\BB$ are both
1-intersecting hypergraphs. Then % For these hypergraphs
 their union $\HH=\AA\cup\BB$ can be considered as a ``geometry'' where two lines intersect in at most one point, and every line has exactly one parallel line. %(Type 2 SPS)

\begin{restatable}{thm}{thmhup}\label{hup}
Assume that $\SS$ is an $(n,n)$-\restricted 1-cross intersecting SPS of size $m$
such that both $\AA$ and $\BB$ are 1-intersecting. Then $m \le {n\choose 2}+1$ for $n>2$.
If $n\ge 4$ and equality holds, then %we have equality,
 $\HH$ is $n$-uniform and $n$-regular $(|A_i|=|B_i|=n$ for $i=1,\dots,m$ and $d_\AA(v)=d_\BB(v)=n)$.
\end{restatable}

In Section~\ref{lboundlin} we give constructive lower bounds.  Constructions~\ref{lboundlin}.1, \ref{lboundlin}.2 and \ref{lboundlin}.3  show that the upper bounds in this subsection are asymptotically the best possible.

\subsection{1-cross intersecting SPS and clique partitions of graphs}\label{cliquepart}
The notion of $1$-cross intersecting SPS is closely related to the concept of clique and biclique partitions. A {\em clique partition} of a graph $G$ is a partition of the edge set of $G$ into complete graphs. Similarly, a {\em biclique partition} of a bipartite graph $B$ is a partition of the edge set of $B$ into complete bipartite graphs (bicliques). The minimum number of cliques (bicliques) needed for the clique (or biclique) partitions are well studied, see, for example~\cite{GMW}. Our problem relates to another parameter of clique (biclique) partitions. The {\em \width } of a clique (biclique) partition of a graph (bipartite graph) is the minimum $s$ such that every vertex of the graph (bipartite graph) is in at most $s$ cliques (bicliques).
Let $T_{2m}$ be the {\it cocktail party graph}, i.e., the complete graph $K_{2m}$ from which a perfect matching is removed. Let $B_{2m}$ be the bipartite graph obtained from the complete bipartite graph $K_{m,m}$ by removing a perfect matching.

Assume that $\SS$ is an $(n,n)$-\restricted   $1$-cross intersecting SPS of   size $m$, and $\HH=\AA\cup\BB$.
The dual of this hypergraph, $\HH^*$, has vertex
set $$V^*=\{x_1,\dots,x_m,y_1,\dots,y_m\}$$ where $x_i,y_i$ correspond to
$A_i,B_i$. The hyperedges of $\HH^*$ correspond to vertices of $\HH$. Since
$|A_i\cap B_j|=1$ for $i\ne j$, every pair $x_i,y_j$ for $i\ne j$ is covered
exactly once by a hyperedge of $\HH^*$. On the other hand, $|A_i\cap
B_i|=0$ for every $i$ so the pairs $x_i,y_i$ are not covered by any
hyperedge of $\HH^*$. Thus the %subgraphs of $B_{2m}$
complete graphs induced by the hyperedges of $\HH^*$ form a biclique
partition of \width $n$ of the bipartite graph $B_{2m}$.

If we have the additional assumption that $\AA$ and $\BB$ are both
1-intersecting then the pairs $x_i,x_j$ and the pairs $y_i,y_j$ are also covered exactly once by the
hyperedges of $\HH^*$.  Thus in this case the complete graphs induced by the hyperedges of $\HH^*$ form a clique
partition of \width $n$ of the cocktail party graph $T_{2m}$.

The above argument gives the following.

\begin{restatable}{thm}{thmtrans}\label{trans}
The maximum $m$ such that $B_{2m}$ has a biclique partition of \width $n$ is
equal to the maximum size of an $(n,n)$-\restricted $1$-cross intersecting SPS.  The maximum $m$ such that $T_{2m}$ has a clique partition of
\width $n$ is equal to the maximum size of an $(n,n)$-\restricted 1-cross
intersecting SPS in which $\AA$ and $\BB$ are also $1$-intersecting.
\end{restatable}

%Theorem \ref{trans} can be used to find condition for equality in Theorem \ref{hup}.

%\begin{prop}\label{pp} For $n\ge 4$ Theorem \ref{hup} gives equality if and only if $T_{n(n-1)+2}$ has a clique partition into $n(n-1)+2$ cliques.
%%%% were also investigated by Lamken, Mullin, and Vanstone~\cite{LMV} (under the name of `twisted projective planes').

\section{Notation and general setting}\label{notgen}

Let $a, b$ positive integers and $I_A, I_B, I_{\mathrm{cross}}$ three sets of non-negative integers.
We denote by $m(a, b, I_A, I_B, I_{\mathrm{cross}})$ the maximum size $m$ of a
cross intersecting SPS $\SS$ with the following % additional
 conditions.
 % restrictions.
 % that is, upper bounds on cardinalities.

\begin{enumerate}[i)]
\item $A_i\cap B_i=\emptyset$ for every $1\le i \le m$,
\item $|A_i|\le a$ for every $1\le i\le m$,
\item $|B_i|\le b$ for every $1\le i\le m$,
\item $|A_i\cap A_j|\in I_A$ for every $1\le i\ne j\le m$,
\item $|B_i\cap B_j|\in I_B$ for every $1\le i\ne j\le m$,
\item $0<|A_i\cap B_j|\in I_{\mathrm{cross}}$ for every $1\le i\ne j\le m$.
\end{enumerate}

\goodbreak

To avoid trivialities we always suppose that $0\not\in I_{\mathrm{cross}}$, also that $m\geq 2$.
If a constraint in iv)--vi) is vacuous (i.e., either $\{ 0,1, \dots, a\} \subseteq I_A$ or $\{ 0,1, \dots, b\} \subseteq I_B$ or $\{ 1, \dots, \min \{ a,b\}\} \subseteq I_{\mathrm{cross}}$)
then we use the symbol $*$ to indicate this.
With this notation Bollob\'as' theorem~\cite{BO} states     $$m(a, b, *, *, *)={a+b\choose a},$$
and our Theorem~\ref{2n} states (for $n\geq 4$)
$$m(2, n, *, *, 1)=\left(\left\lfloor\frac{n}{2}\right\rfloor+1\right)\left(\left\lceil\frac{n}{2}\right\rceil+1\right).$$
In the rest of the results we deal with
   the case $a=b=n$ and use the
abbreviation of placing $n$ as an index
$$
 m_n(I_A, I_B, I_{\mathrm{cross}}):=m(n, n, I_A, I_B, I_{\mathrm{cross}}).$$

Since in this paper the main results are about linear hypergraphs, we will
have $I_A$ (and also $I_B$) is either $\{ 0,1\}$ ($\AA$ is a linear hypergraph), or $\{ 1\}$
 ($\AA$ is a 1-intersecting hypergraph), or $*$.
Instead of writing $I_A=\{ 1\}$ we write `$\inte$',
 instead of $I_A=\{ 0, 1\}$ we write `$\lii$', and
 for $I_{\mathrm{cross}}= \{ 1\}$ we use just `1' (as we did above).

% Also we introduce shorthands to indicate these sets, as `1' or `$=\! 1$'.
% For example $m(a,b,=\! 1,*,1)$ denotes the maximum size $m$ of a
% cross intersecting SPS $\SS$ with the following additional restrictions.
% For every distinct $i,j\in \{1,\ldots,m\}$
% we have $|A_i|\le a$, $|B_i|\le b$, $|A_i\cap A_j|= 1$,
% $|B_i\cap B_j|\le \infty$ and $|A_i\cap B_j|\le 1$.

% We also consider restricting the cardinalities to exact values in
% conditions iii) and iv). In these cases we use the equality sign. For
% example, $m(a,b,=\!1,1,*)$ stands for the maximum size of a
% cross intersecting SPS $\SS$ with the
% restrictions that for every $i\ne j$ we have
% $|A_i|\le a$, $|B_i|\le b$, $|A_i\cap A_j|= 1$,
% $|B_i\cap B_j|\le 1$ and $|A_i\cap B_j|\ne 0$.

Adding more restrictions can only decrease the maximum size, so we have
% \beq   \label{eq2}
% m_n(=\!1, =\!1, 1)\le m_n(=\!1, 1, 1)\le m_n(1, 1, 1)\leq m_n(1, *, *).
%   \eeq
% $$\le m_n(1, 1, *)\le m_n(1, *, *)\le m_n(*, *, *) = {2n\choose n}.$$
\beq\label{eq22}
m_n(\inte, \inte, \cro)\le m_n(\inte, \lii, \cro)\le m_n(\lii, \lii, \cro).
  \eeq

In fact, we examined all 18 cases for $m_n(I_A, I_B, I_{\mathrm{cross}})$ where $I_A$ and $I_B$  are chosen from $\{ 1\}$, $\{ 0, 1\}$, or $*$
% in $\{=1, 1, *\}$ and the third argument is in $\{1,*\}$,
and $I_{\mathrm{cross}}$ is either $\{ 1\}$ or $*$.
% by symmetry there are twelve possibilities.
By symmetry they define twelve functions.
Summarizing our results,
$m_n(*, *, 1)$ and $m_n(*, *, *)$ are exponential as a function of $n$,
the other cases are polynomial.
Three of them, mentioned in~\eqref{eq22}, are asymptotically $\frac{1}{2}n^2$
while the other seven are asymptotically $n^2$.

Several problems under assumptions similar
to $1$-cross intersecting SPS % and their refinements
 have been studied before, see, e.g.,~\cite{Blok,CGYLT,FU,TU1} and more recently in~\cite{GKMNPTX,SW}.

\section{$1$-cross intersecting SPS -- proofs}\label{1-cross}

\propmult*

%\noindent {\bf Proof of Proposition \ref{mult}. }
\begin{proof}
We have to show that $$m(a_1\!+\!a_2, b_1\!+\!b_2, *, *, 1)\ge m(a_1, b_1, *,
*, 1)\cdot m(a_2, b_2, *, *, 1).$$
 Consider $t=m(a_2, b_2, *, *, 1)$ pairwise disjoint
  ground sets $V_1,\dots,V_t$ and for all $i\in [t]$ a copy $(\AA_i,\BB_i)$
  of a construction giving an $(a_1,b_1)$-\restricted  $1$-cross intersecting
 SPS  of size $s$ such that
 $\AA_i=\{A_{i,1},\dots,A_{i,s}\},\;\BB_i=\{B_{i,1},\dots,B_{i,s}\}$,
 where $s=m(a_1, b_1, *, *, 1)$.
 Let $(\AA,\BB)$ be  a copy of an $(a_2,b_2)$-\restricted
 $1$-cross intersecting SPS  of size $t$ on the ground set $V$ such that $\AA=\{A_1,\dots,A_t\},\; \BB=\{B_1,\dots,B_t\}$, where
$V$ is disjoint from all $V_i$-s. For any $1\le i\le t,\; 1\le j\le s$  define  $$A'_{i,j}=A_{i,j}\cup A_i, \; B'_{i,j}=B_{i,j}\cup B_i.$$  The pairs $(A'_{i,j},B'_{i,j})$ form a $1$-cross
intersecting SPS such that $|A'_{i,j}|\le a_1+a_2$ and
$|B'_{i,j}|\le b_1+b_2$.
\qed
%%% \end{proof}

\proplinf*

%\noindent {\bf Proof of Proposition \ref{linf}. }
%%% \begin{proof}
\proof
 Let ${\mathbf{a}}_i$ (resp.\ ${\mathbf{b}}_i$) denote the characteristic vector of $A_i$
  (resp.\ $B_i$), i.e. ${\mathbf{a}}_i(v)=1$ for $v\in V$ if
  and only if $v\in A_i$. Otherwise the coordinates are $0$.
  Suppose that
$$ \sum_{i=1}^m \lambda_i {\mathbf{a}}_i = {\mathbf{0}}.$$ Take the dot product of both sides
  of this equation with ${\mathbf{b}}_j$. Since
  $|A_i\cap B_j|=1$  for $i\ne j$ and
  $|A_i\cap B_j|=0$ for $i=j$, we get that
  $$\left(\sum_{i=1}^m \lambda_i\right)-\lambda_j=0.$$
  Adding
  these for all $j$ yields
  $(m-1) \left(\sum_{i=1}^m \lambda_i\right)=0$. Consequently
  (using $m>1$) $\sum_{i=1}^m \lambda_i=0$. Thus $\lambda_j=0$ for all
  $j$. \end{proof}

\smallskip
\thmtwon*

%\noindent {\bf Proof of Theorem \ref{2n}. }
\begin{proof}
Let $\SS$ be a $(2,n)$-\restricted
$1$-cross intersecting SPS of size $m$.
It is convenient to assume that $\AA$ is two-uniform (a graph without multiple edges)
  and  $\BB$ is an $n$-uniform hypergraph. (For smaller sets dummy vertices
  can be added).

 % First we remark that for a $(2,n)$-\restricted  cross intersecting SPS,
 % the property of being
 % $1$-cross intersecting is equivalent to the Sperner property (no hyperedge
 % of $\HH$  can contain another hyperedge).

Consider the simple graph $\AA$.

\begin{lemma}\label{evenc} If $\AA$ contains a cycle then $m\le 2n+1$.
\end{lemma}
\begin{proof}
The $n$-set $B_i$ must be an
independent transversal for all edges other than $A_i$ (i.e., intersects all edges of $\AA$ except $A_i$ but does not contain any edge of $\AA$)  and disjoint from the edge $A_i$.
Suppose that the graph $\AA$ contains an even cycle with edges $A_1=(x_1,x_2),A_2=(x_2,x_3).\dots A_{2k}=(x_{2k},x_1)$. Since $B_1$ is an independent transversal for all edges other than $A_1$, we have $x_3\in B_1$ which implies $x_4\notin B_1$, and so on, finally $x_{2k}\notin B_1,\; x_1\in B_1$ contradicting $A_1\cap B_1=\emptyset$. Thus $\AA$ has no even cycles.

If there is an odd cycle $C$ with $k$ vertices, it cannot contain a diagonal, since any diagonal would create an even cycle, contradicting the previous paragraph. If there is an edge $A_i$ with exactly one vertex, say $x_1$  on $C$, then the argument of the previous paragraph implies $x_2\in B_i,\; x_3\notin B_i,\dots, x_1\in B_i$, contradiction. Also, if there is an edge $A_i$ with no vertex on $C$ then $B_i$ must intersect all edges of $C$ so it cannot be an independent transversal. Thus in this case $m\le |C| \le 2n+1$.
\end{proof}

Assume next that $\AA$ is an acyclic graph.

\begin{lemma}\label{maxfi} Assume that $T\subseteq \AA$ is a non-star tree component with $t$ edges. Then $$\max_{A_i\in T} |B_i\cap V(T)|\ge \left\lceil{t\over 2}\right\rceil.$$
\end{lemma}

\begin{proof}
Let $P=x,y,z,z_2, \dots$ be a maximal path of $T$, set
$A_1=\{x,y\},A_2=\{y,z\}$. Let $S\subseteq V(T)$ the set of leaves connected to
$y$.
%be the star whose edges are the leaves of $T$ adjacent to $y$.
  Note that $t\geq 3$, $|V(T)|=t+1$, $\;N_T(y)=S\cup \{ z\}$ and $x\in S$.
Then $B_1\cap V(T)$ is the set $X$ of vertices with odd distance from $y$ in the tree $T-x$.  On the other hand, $B_2\cap V(T)$ is the set $X'=S\cup D$ where
 %$C$ is the set of leaves of $S$  and
$D$ is the set of vertices with odd distance from $z$ in the tree $T-(S\cup \{
y\})$.  Then $|X|+|X'|=t+|S|-1\ge t$. % where $s\ge 0$ is the number of leaves of $S$ minus one.
Therefore
\begin{equation*}
  %%% \max_{A_i\in T} |B_i\cap V(T)|\ge
  \max\{|B_1\cap V(T)|, |B_2\cap V(T)|\}
  =\max\{ |X|,
|X'|\} \ge \left\lceil{t\over 2}\right\rceil.
\tag*{\qedhere}
\end{equation*}
\end{proof}

Assume that there is a non-star tree component $T$ in $\AA$ with $t$ edges, $A_1,\dots,A_t$,
$(t\ge 3)$.
We define another $(2,n)$-\restricted
$1$-cross intersecting SPS $(\mathcal{A'},\mathcal{B'})$ of size $m$. Let ${{\cal{A}}'}$ be the graph defined by replacing $T$ with $S$, where $S$ is the union of two vertex disjoint stars $S_1$ and $S_2$ with centers $s_1,s_2$ having $\left\lceil{t\over 2}\right\rceil$ and  $\left\lfloor{t\over 2}\right\rfloor$ edges, respectively. We keep all edges of the other components of $\AA$, i.e., $\AA'= (\AA\setminus E(T))\cup E(S)$.

For $i=1,\dots,t$ in case of $A_i'\in E(S_\alpha)$ let $C_i$ be the complement of $A_i'$ in the star $S_\alpha$ %%% containing $A_i'$
together with the center of the other star of $S$, i.e., $C_i= \left(V(S_\alpha)\setminus A_i'\right)\cup \{ s_{3-\alpha}\}$. Note that $|C_i|$ is either $\lfloor{t\over 2}\rfloor$ or $\lceil{t\over 2}\rceil$.
%%% For $i>t$ let $C_i:= \{ s_1, s_2\}$.
According to Lemma~\ref{maxfi} there is a hyperedge, say $B_1$, with
 $|B_1\cap V(T)|\ge \left\lceil{t\over 2}\right\rceil$.
Define ${{\cal{B}}'}$ as follows.
$$B_i':=\left\{\begin{array}{ll}  C_i\cup (B_1\setminus V(T))  & \text{for } 1\leq i\leq t, \\
                \{ s_1, s_2\}\cup (B_i\setminus V(T))& \text{for } i>t.
                \end{array}\right.
                $$

% For $i>t$ we define  $$B_i'=\{s_1,s_2\}\cup (B_i\setminus (B_i\cap T)).$$
\begin{claim}\label{allstars} $({{\cal{A}}'},{{\cal{B}}'})$  is  a $(2,n)$-\restricted $1$-cross intersecting SPS of size $m$.
\end{claim}
\begin{proof} It is clear that $({{\cal{A}}'},{{\cal{B}}'})$ is a $1$-cross intersecting SPS of size $m$. To prove that it is $(2,n)$-%%% !!! \restricted,
bounded, assume first that
$1\le i\le t$. Then
\begin{equation*}
   |B_i'|=%%%|B_i'\cap S|+|B_i'\cap (V({\cal{A}'})\setminus S)|=
   |C_i|+|B_1\setminus V(T))| %%% \\
   \le \left\lceil{t/2}\right\rceil+\left( |B_1|-\left\lceil{t/2}\right\rceil \right)=|B_1|\leq n.
  \end{equation*}
If $i>t$, we have
$$|B_i'|=%%% |B_i'\cap S|+|B_i'\cap (V({\cal{A}'})\setminus S)|=
          2+|B_i\setminus V(T)|\le |B_i\cap V(T)|+|B_i\setminus V(T)|\leq n,$$
where the inequality $2\le |B_i\cap V(T)|$ holds because $T$ is not a star.
\end{proof}
%\qed

Applying Claim \ref{allstars} repeatedly, we may assume that all components of $\AA$ are stars, $S_1,\dots,S_k$,
where $S_i$ has $t_i\ge 1$ edges. For any edge $A_j\in S_i$, $n\ge |B_j|=t_i-1+k-1$. Adding these inequalities for $i=1,\dots,k$, we obtain that $kn\ge m-2k+k^2$ which leads to $k(n+2-k)\ge m$.
% Using the geometric - arithmetic mean inequality, we get
Hence
$$m\leq  k(n+2-k)\leq   \Bigl(\left\lfloor{n\over 2}\right\rfloor+1\Bigr)\Bigl(\left\lceil{n\over 2}\right\rceil+1\Bigr).
%\ge {(n+2)\over 2}^2\ge \ge m.
$$
Taking together the bounds for odd cycles and acyclic graphs, we get that
$$m\le \max \left\{2n+1,\; \Bigl(\left\lfloor{n\over 2}\right\rfloor+1\Bigr)\Bigl(\left\lceil{n\over 2}\right\rceil+1\Bigr)\right\}.$$
For $n=2,3$ the first term is larger, for $n=4$ they are equal, and for $n\ge5$
the second term takes over. This proves the upper bound for $m$.

The matching lower bound for $n\ge 4$ comes from Proposition~\ref{mult}
applied to the standard construction with values $(1,\lceil{n\over2}\rceil)$
and $(1,\lfloor{n\over2}\rfloor)$. For $n=2$ the hypergraph $\WWW(2,2)$ works
(defined in Subsection \ref{expsize}).
For $n=3$ we can define $\WWW(2,3)$ as the pairs
$\bigl(\{i, i\!+\!1\}, \{i\!+\!2, i\!+\!4, i\!+\!6\}\bigr)$ taken$\pmod 7$.
\end{proof}

\smallskip

\section{1-cross intersecting linear SPS -- upper bounds}\label{linear}
For $v\in V$ we
denote by $d_\AA(v),\; d_\BB(v),\; d_\HH(v)$ the degree of $v$ in
the hypergraphs $\AA, \BB, \HH$,  respectively.

\propujup*

%\noindent {\bf Proof of Proposition \ref{ujup}. }
\begin{proof}
Our first observation here is the following.
% Suppose that
%  $\SS$ is an $(n,n)$-\restricted  cross intersecting SPS of size $m\ge n^2+n+2$,
%  where $\AA$ is linear.
\begin{claim} \label{degcl} $d_\AA(v)\le n+1$ for each vertex $v$.
\end{claim}
\begin{proof}  Suppose $v\in A_1\cap\ldots\cap A_{n+2}$.
Then $v\not\in B_i$ for $i\le n+2$ and in
  $\bigcup_{i=1}^{n+2}A_i\setminus\{v\}$
  the sets $A_i'=A_i\setminus\{v\}$ are pairwise disjoint.
  The set $B_{n+2}$ must intersect each $A'_1,\ldots,A'_{n+1}$ which is impossible.
\end{proof}
Consider $B_{n^2+n+2}$. For $1\le i\le n^2+n+1$ the set $A_i$
  intersects $B_{n^2+n+2}$, so there is a vertex $v\in B_{n^2+n+2}$ with
  $d_A(v)>n+1$, a contradiction.
\end{proof}

\thmgup*

%\noindent {\bf Proof of Theorem \ref{gup}. }
\begin{proof}
Suppose that $(\AA,\BB)$ is an $(n,n)$-\restricted  % XXX ??? type 1
 1-cross intersecting SPS of size $m$ such that both $\AA$ and $\BB$ are linear hypergraphs.
We have  $m_2(\lii, \lii, 1)\leq 5$ by Theorem~\ref{2n} so we may suppose that $n\geq 3$.
If $m\leq 2n+2$ then there is nothing to prove, so from now on, we may suppose that $m\geq 2n+3$.

We claim that for every $v\in V$,
% \beq\label{eq3}
  $d_\AA(v)$, $d_\BB(v)\le n$.
  % \eeq
Indeed, $d_\AA(v)\leq n+1$ (and in the same way $d_\BB(v)\le n+1$) is obvious from Claim~\ref{degcl}.
Suppose $d_\AA(v)\ge n+1$, say $v\in A_1\cap \dots \cap A_{n+1}$ then $m>2n+2\geq d_\AA(v)+d_\BB(v)$ so there is a pair $A_i,B_i$ with $i>n+1$ such that $v\notin A_i\cup B_i$. Thus $B_i$ cannot intersect all $A_j$-s containing $v$, proving the claim.

Since $\SS$ is 1-cross intersecting we have $\sum_{v\in B_i} d_\AA(v)=m-1$ for each $B_i$.
Adding up these $m$ equations we get
\beq\label{eq4}
\sum_v  d_\AA(v)d_\BB(v)=m^2-m.
  \eeq

Let $\AA_i$ be the set of $A_j$-s that intersect $A_i$ and different from $A_i$.
Our crucial observation is that if $A_i$ and $A_j$ do not intersect then
\beq\label{eq5}
   |\AA_i|+|\AA_j|\leq n^2.
  \eeq
Indeed, the left hand side of (\ref{eq5}) equals to
$\sum_{\ell: \ell\neq i,j}  |A_\ell\cap (A_i\cup A_j)|$.  For two disjoint sets $X,Y$ we say that a pair $(x,y)$ {\em joins} $X,Y$ if $x\in X,y\in Y$.
For $\ell \neq i,j$ we have $|A_\ell\cap (A_i\cup A_j)|\leq 2$.
In case of  $|A_\ell\cap (A_i\cup A_j)|= 2$ we select two pairs $(x,y),(x',y')$ joining $A_i,A_j$, namely
 $(x,y)=A_\ell\cap (A_i\cup A_j)$ and $(x',y')=B_\ell\cap (A_i\cup A_j)$.
In case of  $|A_\ell\cap (A_i\cup A_j)|= 1$ we select one pair $(x,y)$ joining  $A_i,A_j$, namely
 $(x,y)=B_\ell\cap (A_i\cup A_j)$.
These pairs are distinct because
$$|A_{\ell}\cap B_{\ell'}|\le 1,|A_{\ell}\cap A_{\ell'}|\le 1, |B_{\ell}\cap B_{\ell'}|\le 1.
  $$
Since there are $n^2$ pairs between $A_i$ and $A_j$ we obtain that
  $\sum_{\ell: \ell\neq i,j}  |A_\ell\cap (A_i\cup A_j)|\leq n^2$, completing the proof of~(\ref{eq5}).
% This implies~\eqref{eq5}.

If $A_i\cap A_j=\{v\}$  then %, with a slightly more complicated argument,
   we will prove that
\beq\label{eq6}
   |\AA_i|+|\AA_j|\leq (n-1)^2+d_\AA(v)+d_\BB(v) \leq n^2+1.
  \eeq
Indeed, as before,
\[ |\AA_i|+|\AA_j|=\sum_{\ell: \ell\neq i}  |A_\ell\cap A_i| +\sum_{\ell: \ell\neq j}  |A_\ell\cap A_j|.
  \]
For every  $\ell \neq i,j$ we select (at most) two pairs joining $A_i\setminus \{ v\}$ to $A_j\setminus \{ v\}$, namely
$A_\ell \cap ((A_i\setminus \{ v\})\cup (A_j\setminus \{ v\}))$ and
 $B_\ell \cap ((A_i\setminus \{ v\})\cup (A_j\setminus \{ v\}))$.
In this way we selected at least $ |A_\ell\cap A_i|+|A_\ell\cap A_j|$ distinct pairs except if $v\in A_\ell \cup B_\ell$.
In the latter case we still have selected at least $ |A_\ell\cap A_i|+|A_\ell\cap A_j|-1$ pairs.
So the left hand side of~\eqref{eq6} is at most the number of pairs joining $A_i\setminus \{ v\}$ to $A_j\setminus \{ v\}$ plus
 $d_\AA(v)+d_\BB(v)$. This completes the  proof of~\eqref{eq6}.

Next we prove that
\beq\label{eq7}    \sum_{v\in V} d_\AA(v)^2 \leq m\left(\frac{1}{2}n^2 + n + \frac{1}{2}\right).
  \eeq
Add up inequalities~\eqref{eq5} and~\eqref{eq6} for all $1\leq i<j\leq m$
\[
  \frac{1}{m-1}\sum_{1\leq i<j\leq m} |\AA_i|+|\AA_j| \leq  \frac{1}{m-1}{m \choose 2} (n^2+1)= m\left(\frac{1}{2}n^2 + \frac{1}{2}\right).
\]
Here the left hand side is
\[
    \sum_{1\leq i\leq m} |\AA_i|=  \sum_{1\leq i\leq m}\left( \sum_{v\in A_i} (d_\AA(v)\!-\!1)\right)
    =\sum_{v\in V} \left(d_\AA(v)^2-d_\AA(v)\right)= \left(\sum_{v\in V} d_\AA(v)^2\right) -mn.
\]
  The last two displayed formulas yield~\eqref{eq7} and equality can hold only if ~\eqref{eq5} was not used. Note that similar upper bound must hold for $ \sum_{v\in V} d_\BB(v)^2 $, too.

Apply~\eqref{eq7} to $\AA$ and to $\BB$ and subtract the double of~\eqref{eq4}. We obtain
\begin{multline*}
0\leq \sum_{v\in V} (d_\AA(v)-d_\BB(v))^2= \sum_{v} d_\AA(v)^2 + \sum_{v} d_\AA(v)^2 -2\sum_v  d_\AA(v)d_\BB(v)
\\
\leq 2m\left( \frac{1}{2}n^2 + n + \frac{1}{2}\right) -2m(m-1)= 2m \left( \frac{1}{2}n^2 + n + \frac{3}{2}-m\right).
   \end{multline*}
This implies $m\leq \frac{1}{2}n^2 + n + \frac{3}{2} $.
As a last step we show that this inequality is strict completing the proof of the upper bound on $m$.
Indeed, equality can hold only if~\eqref{eq5} was never used to $\AA$ neither to $\BB$.
This implies that $\AA$ and $\BB$ are 1-intersecting and because of (\ref{eq6}) there exists a $v$ with $d_\AA(v)=d_\BB(v)=n$.
%%% But this easily leads to a contradiction. Indeed, if
Suppose
$$v\in A_1\cap \dots \cap A_n\cap B_{n+1}\cap \dots\cap B_{2n}.$$
Then
$A_{n+1}\cap B_{n+2}=\emptyset$ because  $A_{n+1}\cap B_i$, $B_{n+2}\cap A_i$ are nonempty for $i=1,\dots,n$.
This contradicts the $1$-intersection property.
\end{proof}

\thmhup*

%\noindent {\bf Proof of Theorem \ref{hup}. }
\begin{proof}
% Suppose that $\SS$ is an $(n,n)$-\restricted 1-cross intersecting SPS of size $m$
%   such that both $\AA$ and $\BB$ are 1-intersecting, and recall
Recall  that $\HH=\AA\cup \BB$.
First, consider the case when there exists a vertex $v$ with $d_\HH(v)\geq n+1$, say $v\in A_i\cup B_i$ for $i\in \{1, 2, \dots, n+1 \}$.
Then one of the members of $\{ A_{n+2}, B_{n+2}\} $ does not cover $v$, say, $v\notin A_{n+2}$. Then $A_{n+2}$ cannot intersect all members of $\{ A_i, B_i\}_{1\leq i\leq n+1}$ containing $v$, a contradiction. So in this case $m=n+1$ and we are done.

From now on, we may suppose that $m> n+1$, and $d_\HH(v)\leq n$ for all $v\in V $.
Since only $B_1$ is disjoint from $A_1$ we get
\[ 2m=|\HH|= 2+ \sum_{v\in A_1} (d_\HH(v)-1)\leq 2+ n(n-1).
\]
and we conclude that $m\le {n\choose 2}+1$. If $n\ge 4$ and equality holds, then all vertices of $A_1$ (and of all other hyperedges) must have degree $n$. \end{proof}

\section{Constructing cross intersecting linear \\
  hypergraphs}\label{lboundlin}

Here we give constructions of large cross intersecting SPS-s such that $\AA$ is an intersecting
linear hypergraph.
Constructions~\ref{lboundlin}.1 and \ref{lboundlin}.2 show that
%\begin{equation}
  \begin{gather} n^2-o(n^2)\le m_n(\inte, \inte, *), \label{eq9}
   \\n^2-o(n^2)\le m_n(\inte, *, 1).\label{eq10}
  \end{gather}
%\end{equation}
Since the right hand sides of these inequalities are bounded above by
$m_n(\lii, *, *)$ (which is at most $n^2+n+1$), Proposition~\ref{ujup} is asymptotically the best possible.
Construction~\ref{lboundlin}.3 shows that
\begin{equation}\label{eq11}
  \frac{1}{2}n^2-o(n^2)\le m_n(\inte, \inte, 1). % \le m_n(\lii,\lii,1).
  \end{equation}
Hence Theorems~\ref{gup} and~\ref{hup} are also asymptotically the best possible.

%  The outline of the constructions is the following.
We use that the function $m_n(I_A, I_B, I_{\rm cross})$ is monotone increasing in $n$ so we have to make constructions only for a dense set of special values of $n$.

Beyond Bertrand's postulate (for each real $x>1$ there always exists a prime $p$ with $x< p< 2x$)
we need Hoheisel's theorem~\cite{HOH} about the density of primes:
There are constants $x_0$ and $0.5\le \alpha<1$ such that
    for all $x\ge x_0$ the interval
\begin{equation}\label{HOHcor}
   [x-x^{\alpha}, x] \,\,\, \textrm{contains a prime number}.
\end{equation}
The currently known best $\alpha$ is $0.525$ by Baker, Harman and Pintz~\cite{BHP}.

\subsection{Building blocks: double stars and affine planes}\label{ss41}
The vertex set of a {\em double star of size $s$} consists of
  $\{v_{i,j}\;|\; 1\le i,\; j\le s,\; i\ne j\}$ and two
  additional special vertices $w_a$ and $w_b$. Define for $i=1,\ldots,s$ sets
  $A_i=\{w_a\}\cup\{v_{i,j}\;|\; 1\le j\le s,\; j\ne i\}$ and
  $B_i=\{w_b\}\cup\{v_{j,i}\;|\; 1\le j\le s,\; j\ne i\}$.
Then $(\AA,\BB)$ is a %type 2
  1-cross intersecting
SPS of size $s$ containing $s$-element sets
 such that both $\AA$ and $\BB$ are 1-intersecting.
 The double star shows that $m_n(\inte, \inte, 1)\geq n$ for all $n$
 (consequently, $m_n(\inte, \inte, *)\geq n$ and $m_n(\inte, *, 1)\geq n$).

The affine plane AG$(2,q)=(P, \cL)$ is a $q$-uniform hypergraph
 with a $q^2$ element vertex set $P$, such that each edge $L\in \cL$
 (called {\em line}) has  $q$ vertices ({\em points}),
 and $\cL$ can be split into $q+1$ parts
$\cL =\cL_1\cup \cL_2\cup \dots \cup \cL_{q+1}$
 (directions or parallel classes of lines)
 such that each parallel class
contains $q$ lines, $\cL_\delta= \{L_{1, \delta}, \dots, L_{q, \delta}\}$,
the members of a parallel class are pairwise disjoint,
but two lines from distinct classes always meet in a single point.
It is known that an AG$(2,q)$ exists if $q$ is prime. % or a power of a prime.

In the next subsection we give three different (but similar) constructions to
prove the lower bounds~\eqref{eq9}--\eqref{eq11}. Each construction will use
an associated Extension \emph{twice},
where an Extension starts with a weaker construction of
the same type and combine it with AG$(2,q)$ for getting a stronger construction.
In the following $p$ and $q$ will always denote odd primes.

\subsection{Extensions of the affine plane}\label{ss42}

\medskip
\noindent {\bf Extension I. }
Let $( \AA', \BB' )$ be a  cross-intersecting SPS of size at least $q$.
For each $1\leq \delta\leq q+1$ take a new copy of  $( \AA', \BB' )$
 so that the ground sets of the $q+1$ copies are pairwise disjoint and also disjoint from  AG$(2,q)$.
For $i=1,\ldots,q$ let $(A'_{i,\delta}, B'_{i,\delta})$ be the disjoint pairs in the $\delta$th copy.

Define $\CC_1(q, \AA')$ by the $q^2+q$ sets
$A_{i,\delta}:=L_{i,\delta}\cup A'_{i,\delta}$, and define
$\CC_1(q, \BB')$ by the $q^2+q$ sets
$B_{i,\delta}:=L_{i+1,\delta}\cup B'_{i,\delta}$, here $L_{q+2, \delta}:= L_{1, \delta}$.

\begin{claim}\label{cl1}
$(\CC_1(q, \AA'), \CC_1(q,\BB'))$ is a cross-intersecting SPS. If
 $\AA'$ and $\BB'$ are 1-intersecting hypergraphs, then so do $\CC_1(q, \AA')$ and  $\CC_1(q,\BB')$.
\end{claim}
\begin{proof}
Indeed, $A_{i,\delta}\cap B_{j,\gamma}=
(L_{i,\delta}\cap L_{j+1, \gamma}) \cup (A'_{i,\delta}\cap B'_{j,\gamma})$.
This is the singleton $L_{i,\delta}\cap L_{j+1,\gamma}$ for $\delta\ne\gamma$,
 it contains the nonempty set $A'_{i,\delta}\cap B'_{j,\delta}$ for $\delta =\gamma$ and $i\neq j$,
 and it is empty for $({i,\delta})=(j,\gamma)$.

In the case $\AA'$ is 1-intersecting and $({i,\delta})\neq (j,\gamma)$  we get that $A_{i,\delta}\cap A_{j,\gamma}=
(L_{i,\delta}\cap L_{j, \gamma}) \cup (A'_{i,\delta}\cap A'_{j,\gamma})$, a singleton.
\end{proof}

%% \medskip
%%  \noindent{\em Construction~\ref{lboundlin}.1, showing $m_n(\inte, \inte, *)\ge n^2- 10 n^{1+\alpha}$.} \newline
\begin{constr}
We prove \eqref{eq9}, i.e., $m_n(\inte, \inte, *)\ge n^2- 10 n^{1+\alpha}\ge n^2-o(n^2)$.

Claim~\ref{cl1} implies that whenever $q$ is an odd prime and
$m_s(\inte, \inte, *)\geq q$ then
\begin{equation}\label{eq13}
 m_{q+s}(\inte, \inte, *) \geq q^2+q.
  \end{equation}
Since $m_{s}(\inte, \inte, *) \geq s$ by the double star, apply~\eqref{eq13} for $(q,s)=(p,p)$.
We get  $m_{2p}(\inte, \inte, *)\ge p^2+p$ for all primes $p>2$.

Suppose $n>2x_0$.  There is a prime $q$ between
$n-5n^\alpha$ and $n-4n^\alpha$ by~\eqref{HOHcor} and  there is another prime $p$ between $n^\alpha$ and $2n^\alpha$.
Since
  $m_{2p}(\inte, \inte, *)\geq p^2+p> n^{2\alpha}>n > q$ one can apply~\eqref{eq13} with $s:=2p$
\begin{equation*} %\label{eq14}
   m_{n}(\inte, \inte, *)\geq  m_{q+2p}(\inte, \inte, *)
     \geq q^2+q >  n^2- 10 n^{1+\alpha}.
  \end{equation*}
\end{constr}

Note that $|A_{i,\delta}\cap B_{j,\gamma}|$ can be as large as $q+1$ (for $i=j+1$).

\medskip

Next  we prove~\eqref{eq10} and~\eqref{eq11}. The proofs are rather similar to
the one presented above, so we leave out most of the details.

\medskip

\noindent {\bf Extension II. }
Let $( \AA', \BB' )$ be a  1-cross-intersecting SPS of size at least $q-1$.
For each $1\leq \delta\leq q+1$ take a new copy of  $( \AA', \BB' )$
 so that the ground sets of the $q+1$ copies are pairwise disjoint and also disjoint from  AG$(2,q)$.
For $i=1,\ldots,q-1$ let $(A'_{i,\delta}, B'_{i,\delta})$ be the disjoint pairs in the $\delta$th copy.

Define $\CC_2(q, \AA')$ by the $q^2-1$ sets
$A_{i,\delta}:=L_{i,\delta}\cup A'_{i,\delta}$, and define
$\CC_2(q, \BB')$ by the $q^2-1$ sets
$B_{i,\delta}:=L_{q,\delta}\cup B'_{i,\delta}$.

\begin{claim}\label{cl2}
$(\CC_2(q, \AA'), \CC_2(q,\BB'))$ is a 1-cross-intersecting SPS. If
 $\AA'$ is a 1-intersecting hypergraph, then so does $\CC_2(q, \AA')$. \qed
\end{claim}

\begin{constr}
We prove \eqref{eq10}, i.e., $m_n(\inte, *, 1)\ge n^2- o(n)$.

Claim~\ref{cl2} implies that whenever $q$ is an odd prime and
$m_s(\inte, *, 1)\geq q-1$ then
\begin{equation}\label{eq15}
 m_{q+s}(\inte,  *, 1) \geq q^2-1.
  \end{equation}
Since $m_{s}(\inte, *, 1) \geq s$ by the double star, apply~\eqref{eq15} for  $(q,s)=(p,p-1)$.
We get  $m_{2p-1}(\inte, *,1)\ge p^2-1$ for all primes $p>2$.

There is a prime $q$ between
$n-5n^\alpha$ and $n-4n^\alpha$ and  there is another prime $p$ between $n^\alpha$ and $2n^\alpha$.
Since
  $m_{2p-1}(\inte, *, 1)\geq p^2-1 > n^{2\alpha}-1 \geq n > q$ one can apply~\eqref{eq15} with $s:=2p-1$
\begin{equation*} % \label{eq16}
   m_{n}(\inte,  *,1)\geq  m_{q+2p-1}(\inte,  *,1)
     \geq q^2-1 >  n^2- 10 n^{1+\alpha}.
  \end{equation*}
\end{constr}

Note that $\CC_2(q, \BB')$ is \emph{not} linear.

\medskip
\noindent {\bf Extension III. }
Let $( \AA', \BB' )$ be a  1-cross-intersecting SPS of size at least $(q-1)/2$.
For each $1\leq \delta\leq q+1$ take a new copy of  $( \AA', \BB' )$
 so that the ground sets of the $q+1$ copies are pairwise disjoint and also disjoint from  AG$(2,q)$.
For $i=1,\ldots,(q-1)/2$ let $(A'_{i,\delta}, B'_{i,\delta})$ be the disjoint pairs in the $\delta$th copy.

Define $\CC_3(q, \AA')$ by the $(q^2-1)/2$ sets
$A_{i,\delta}:=L_{i,\delta}\cup A'_{i,\delta}$, and define
$\CC_3(q, \BB')$ by the $(q^2-1)/2$ sets
$B_{i,\delta}:=L_{i+(q-1)/2,\delta}\cup B'_{i,\delta}$.

\begin{claim}\label{cl3}
  $(\CC_3(q, \AA'), \CC_3(q,\BB'))$ is a 1-cross-intersecting SPS.  If
 $\AA'$ and $\BB'$ are 1-intersecting hypergraphs, then so do $\CC_3(q, \AA')$ and  $\CC_3(q,\BB')$.\qed
\end{claim}

\begin{constr}
We prove \eqref{eq11}, i.e., $m_n(\inte, \inte, 1)\ge \frac{1}{2}n^2-o(n^2)$.

Claim~\ref{cl3} implies that whenever $q$ is an odd prime and
$m_s(\inte, \inte, 1)\geq (q-1)/2$ then
\begin{equation}\label{eq17}
 m_{q+s}(\inte,  \inte, 1) \geq (q^2-1)/2.
  \end{equation}
Since $m_{s}(\inte, \inte, 1) \geq s$
 by the double star, apply~\eqref{eq17} for  $(q,s)=(p,(p-1)/2)$.  We get  $m_{(3p-1)/2}(\inte, \inte,1)\ge (p^2-1)/2$ for all primes $p>2$.

There is a prime $q$ between
$n-5n^\alpha$ and $n-4n^\alpha$ and  there is another prime $p$ between $n^\alpha$ and $2n^\alpha$.
Since
  $m_{(3p-1)/2}(\inte,  \inte, 1)\geq (p^2-1)/2 > n^{2\alpha}/2 \geq n > q$ one can apply~\eqref{eq17} with $s:=(3p-1)/2$
\begin{equation*} % \label{eq18}
   m_{n}(\inte,  \inte,1)\geq  m_{q+(3p-1)/2}(\inte,  \inte,1)
     \geq \frac{1}{2}(q^2-1) >   \frac{1}{2}n^2- 5 n^{1+\alpha}.
  \end{equation*}

\end{constr}

\section{Conjectures, open problems}

We conjectured~\cite{FGK-arxiv} that there exists a positive $\varepsilon$ such that
  $m_n(*, *, 1)\le(1-\varepsilon){2n\choose n}$ for every $n\ge 2$.
This was proved by Holzman~\cite{Hol} in the following stronger form. If $a, b\geq 2$, then $m(a, b,1)\leq (29/30) {a+b\choose a}$.
More recently Kostochka, McCourt, and Nahvi~\cite{KMN} showed that the factor $29/30$ in this bound can be replaced by $5/6$, which is the best possible since $m(2,2,1)=5$.

Although Constructions~\ref{lboundlin}.1, and~\ref{lboundlin}.3 together with Proposition~\ref{ujup} and Theorem~\ref{gup} show that
$$\lim_{n\to \infty} \frac{m_n(\inte, \inte, 1)}{m_n(\inte, \inte,
  *)}=\lim_{n\to \infty} \frac{m_n(\lii, \lii, 1)}{m_n(\lii, \lii, *)}=\frac{1}{2},$$
we strongly believe that the following % stronger conjecture
 is also true.

\begin{conj}\label{conj_little-o}
   $$\lim_{n\to \infty} \frac{m_n(*, *, 1)}{m_n(*, *, *)}=0.$$
\end{conj}

%% The proof of Proposition~\ref{ujup} gives   $m(a, b, \lii, *, *)\le b^2+b+1$. % XXX ???
%% Since the function $m$ is monotone in $a$ and $b$ (one can add a different new vertex to each member of $\AA$), we get the asymptotic
%%   $m=b^2+O(b)$ for many of the cases we have considered.
%%   We also have a lower bound  for all $a<b$ from Theorem~\ref{2n} since $m(2,
%%   b, *, *, 1)\approx b^2/4$.
%% It would be interesting to investigate all cases and the case $a<b$ as well.

We obtained some tight results for
$m(a, b, I_A, I_B, I_{\mathrm{cross}})$ in the case $a=b$ and also in the case $a=2$.  There is plenty of room for further investigations.

\bigskip

\noindent {\bf Acknowledgment. } We thank to P.\ Frankl,
T.\ H\'eger, and  D.\ P\'alv\"olgyi for fruitful discussions. Helpful remarks of a referee are also appreciated.

%\eject

\end{document}